\documentclass[reqno]{amsart}
\usepackage{amsthm,amsmath,amsfonts,amssymb,amsthm,graphicx,epsf,epstopdf,color,pifont,flafter,bm,amscd,tabu}
\usepackage{graphicx}
\usepackage{subcaption}

\newtheorem {theorem} {Theorem}

\newtheorem {definition}{Definition}
\newtheorem {corollary}{Corollary}
\newtheorem {lemma}  {Lemma}

\newtheorem {remark}{Remark}

 \usepackage{enumitem}

\begin{document}


\title[Topology and isochronicity on Hamiltonian systems]{The topology and isochronicity  on complex Hamiltonian systems with homogeneous nonlinearities}

\author{Guangfeng Dong}

\address{Department of Mathematics, Jinan University,
	Guangzhou 510632, Guangdong province, China}


\email{donggf@jnu.edu.cn}
\author{Jiazhong Yang}
\address{ School of Mathematical Sciences, Peking University,  Beijing, 100871,  China}
\email{jyang@math.pku.edu.cn}


\begin{abstract}
In this paper, we  study the Hamiltonian differential systems with  homogeneous nonlinearity parts on $\mathbb{C}^2$. Firstly, we present a series of topological properties of polynomial Hamiltonian functions, with a particular focus on the characteristics of critical points and non-trivial cycles that vanish at infinity.   Secondly, we use these topological properties to derive a complete set of necessary and sufficient conditions for isochronous centers in this class of systems of any degree. Our method avoids tedious computation of the coefficients of normalization occurring in the usual tools to deal with the isochronicity problem.\end{abstract}






\maketitle

\section{Introduction and the main results}

Consider the following complex planar  Hamiltonian differential systems of the form
\begin{equation}\label{HSH}
\begin{array}{ccrl}
  	\frac{d x}{d t} & =  & H_y(x,y),& 
  	\\ 
\frac{d y}{d t}  &= &  -H_x(x,y),& \quad (x,y)\in\mathbb{C}^2, \quad t\in\mathbb{C},
\end{array}
\end{equation}
where the   Hamiltonian function $H(x,y)$  is a complex polynomial  in its  complex variables  $x$ and $y$ without constant and linear terms, and $H_x(x,y)$ and $H_y(x,y)$ are the partial derivatives of $H(x,y)$ with respect to the variables:  $$H_y(x,y)=\frac{\partial H (x,y)}{\partial y},\ \ H_x(x,y)=\frac{\partial H (x,y)}{\partial x}.$$
In this notation,  the {\it degree} of the system
(\ref{HSH}) is defined as  the  maximum degree of the polynomials  $H_y$ and $ H_x $,
which is  $n$ if the degree of $H(x,y)$ is $n+1$.

It is clear that the origin $O$ is a singular point.
A (non-degenerate) singular point of system \eqref{HSH} corresponds  a (non-degenerate) critical point of Hamiltonian function  $H(x,y)$.  In the case of a non-degenerate critical point, it is  also known as  a complex Morse critical point, which can be described as a point at which the Hessian matrix of $H(x,y)$ is non-degenerate.
In this paper, we assume  that the origin is  non-degenerate.
Thus, up to a change of coordinates, the polynomial $H(x,y)$ can be expressed as $ H(x,y)= xy +h.o.t,$ where $h.o.t$ denotes the terms of higher degree.

Let $A_{H}$ be the set of  atypical values (or non-generic values) of $H(x,y)$ and $H^{-1}(A_{H})$ be the set of points $(x,y)\in \mathbb{C}^2$ such that $H(x,y)$ takes a value in $ A_{H}$. It is well known that $A_{H}$ is a finite set (see, e.g., \cite{Gavri} and references therein) containing all the critical values of $H$.
The polynomial $H(x,y)$ defines a fiber bundle  in the following way: 
\begin{equation*}
	\begin{array}{crcl}
	\mathsf{H}:	 &\mathbb{C}^2-H^{-1}(A_{H}) &\rightarrow &\mathbb{C}-A_{H}, \\
		&(x,y)&\mapsto &H(x,y),
	\end{array}
\end{equation*}
whose fibers are the generic level curves  $L_h=\{(x,y): H(x,y)=h, h\in \mathbb{C}\}$ of the system \eqref{HSH} for $h\not\in  A_H$. For each fiber $L_h $
one can associate a  {\it vanishing cycle} $\gamma_h$ to the critical value $h=0$,
which is a  cycle generating the set of cycles vanishing  at $O$ as  $h\rightarrow 0$ in the first homology group 
$\mathcal{H}_{1}(L_{h},\mathbb{Z})$.  
The vanishing cycle $\gamma_h$  can be characterized by the following purely
topological property:
modulo orientation and the free homotopy deformation on $L_{h}$,
as $h\rightarrow 0$,
the cycle $\gamma_h$ can be represented by a
continuous family of loops on $L_{h}$ of length that tends to zero.
The {\it period function} of system (\ref{HSH}) associated to the origin, denoted by $T(h)$,  is defined as the  integral along $\gamma_h$, i.e., $$T(h)=\oint_{\gamma_h} dt=\oint_{\gamma_h}\frac{dx}{H_y(x,y)},$$ where the 1-form ${dx}/{H_y(x,y)}$, denoted by $\omega$, is called a {\it period 1-form}.

In the case where $T(h)$ is a non-zero constant that is independent of $h$ for all values $h\neq 0$,
 the origin is said to be an \textit{isochronous center}.
This definition coincides with the classical definition of an isochronous center when $(x,y)\in \mathbb{R}^2$ and $t\in \mathbb{R}$.
One of the simplest examples of an isochronous center  can be given by the differential equations $\dot x=y, \ \dot y=-x$ in the real case, since in this
case the period is always $2\pi$, independent of the length of the orbits and $h$. Furthermore,
  each orbit can represent the corresponding vanishing cycle when the system is embedded into $\mathbb{C}^2$ in a natural way.
The aim of this paper is to
 study the necessary and sufficient conditions  for the origin  to be an isochronous center for a class of systems (\ref{HSH}),
which will be specified shortly afterwards.

The study of isochronous centers is a well-established field of research that has attracted significant attention for decades. For the sake of clarity, we will present the Hamiltonian function $H(x,y)$ in the following form:
\begin{equation}
\label{homo}
H(x,y)=xy+  H_{3}(x,y)+ H_{4}(x,y)+\cdots  +H_{n+1}(x,y),\end{equation}
where $ H_{k}(x,y)$  $(k=3,4,\dots, n+1)$  is a homogeneous polynomial of degree $k$ on $\mathbb{C}^2$.
In general, it is extremely challenging to verify whether a system    of a general form (\ref{HSH}) is   isochronous.  
In this aspect, only a limited number of 
results are available, mainly for systems with relatively low degrees or special forms.
Related results together with some discussions for  real systems
can be found in   \cite{Amel,zhang,cima,llibre,Man} and the references therein.

In the above-mentioned Hamiltonian differential systems, there are several papers that concentrate on
the necessary and sufficient conditions for the isochronicity in Hamiltonian systems with homogeneous nonlinearities. These systems can be expressed as follows:
\begin{equation}\label{m}
\begin{array}{ccrl}
  	\frac{d x}{d t} & =  & x+\frac{\partial H_{n+1}}{\partial y}(x,y),& 
  	\\ 
\frac{d y}{d t}  &= &  -y-\frac{\partial H_{n+1}}{\partial x}(x,y),& \quad (x,y)\in\mathbb{C}^2, \quad t\in\mathbb{C},
\end{array}
\end{equation}
with the corresponding Hamiltonian function 
\begin{equation}\label{m-H}
H(x,y)=xy+  H_{n+1}(x,y)=xy+\sum_{j=0}^{n+1}a_{j} x^{n+1-j} y^{j}, \quad  a_j \in \mathbb{C},
\end{equation} in terms of (\ref{homo}).

It is well known that for system (\ref{HSH}), the origin is isochronous
if and only if it is linearizable under an analytic change  of the variables.
In their study of the linearizability problem, 
 the authors of \cite{ChRo} considered the case of $n=2,3$ for $H(x,y)$ of the form (\ref{m-H})
and obtained all the conditions for isochronicity.
 In a recent paper \cite{AGR},  with the help of computer algebra systems, the authors
successfully walked  some  steps  in this direction.
More precisely,    
they   obtained the  necessary and sufficient   isochronicity conditions for
 $n=4, 5, 6, 7$.
Clearly,   for big number $n$,   the calculation will be  tedious and tremendous, and
it is far from trivial, even impossible,  to exhaust all the cases  without
more theoretically powerful techniques.

In this paper, we shall develop a very different approach,  by studying the topology properties of the polynomial  Hamiltonian function \eqref{m-H},
to terminate this problem. For an overview of previous work on the isochronicity problem, which has been approached through the topology of the polynomial Hamiltonian, the reader is directed to reference  \cite{Gavri}.   In the following theorem, we shall  
 provide an explicit expression of   the necessary and sufficient conditions for the isochronicity of  system (\ref{m}) of any degree $n>1$.

\begin{theorem}\label{th-main}
For system (\ref{m}),  the origin is an isochronous center if and only if one of the following two conditions holds:

\begin{enumerate}
  \item[$(\mathrm{I})$] $a_j=0$ for any integer $j$ satisfying $0 \leq j \leq \frac{n+1}{2} $;
  \item[$(\mathrm{II})$] $a_j=0$ for any integer $j$ satisfying  $\frac{n+1}{2} \leq j \leq n +1$.
\end{enumerate}
	
\end{theorem}

\begin{remark}
We feel that it is worthy  to distinguish between the cases where  $n$ is even or  odd in the
 expression  $a_j$ in the theorem for $j={(n+1)}/{2}$. The precise meaning of these relations  is clear, i.e., this  equation only holds for odd number $n$.

Given that the roles played by the variables  $x$  and $y$ are  identical,  it is evident that the two cases presented in the theorem are, in fact, one and the same.  In other words,
the origin is an isochronous center if and only if the function $H_{n+1}(x,y)$ in system  (\ref{m})
takes one of the following two forms
$$
\begin{array}{rcl}
   H_{n+1}(x,y)&= & y^{[\frac{n+1}{2}]+1} \tilde{H}_{[\frac{n}{2}]}(x,y),  \\ 
   H_{n+1}(x,y)&= & x^{[\frac{n+1}{2}]+1} \tilde{H}_{[\frac{n}{2}]}(x,y),
\end{array}
$$
where $\tilde{H}_{[\frac{n}{2}]}(x,y)$ is a homogeneous polynomial of degree $[\frac{n}{2}]$,
and the conventional notation $[x]$ means  the integer part of $x$.

\end{remark}

\begin{remark}
Note that the isochronous centers occurring  in Theorem \ref{th-main} must not be real centers,
which shows a clear difference between the real and complex settings.
Indeed, as shown in \cite{Chr-Dev,Gasull},  real Hamiltonian systems have no isochronous centers
if they possess only homogeneous nonlinearities.
However, our results  indicate that complex isochronous centers for complex
systems of the same form can generically exist.
\end{remark}

The proof of Theorem \ref{th-main}  totally depends on  the topological properties of the fiber bundle  defined by $H(x,y)=xy+H_{n+1}(x,y)$ and the set of singular  points of system  (\ref{m}). The following theorems will provide a detailed description of these properties. For further details on the subject of singular points and fibre bundles of polynomials, please refer to the reference  \cite{AGV1}.

Recall that a singularity  of a germ of a holomorphic function  $f(x,y)$  is said to be {\emph{simple}} at an isolated critical point of $f(x,y)$, if it belongs to one of the classes $\mathbf{A}_k$, $\mathbf{D}_k$, $\mathbf{E}_6,$ $\mathbf{E}_7,$ and $\mathbf{E}_8$ according to the corresponding Coxeter groups (for further details,  please refer to  Chapter 15 of Volume 1 of \cite{AGV1}). In this paper we focus only on the type $ \mathbf{A}_k$. A polynomial function $f(x,y)\in\mathbb{C}^2$ has type $ \mathbf{A}_k$ with $k\geq 1$ at  an isolated critical point, e.g. the origin,   if and only if $f(x,y)$ can be transformed to $z^2 +w^{k+1}$ by a change of coordinates $(x,y)\mapsto (z,w)$ in a neighborhood of the origin. In the case of a singularity of type $\mathbf{A}_1$, it is equivalent to say that the corresponding critical point is non-degenerate. 


With regard to the isolated critical points, the following theorem is proposed. 
\begin{theorem}\label{th-sim-s}
The polynomial $H(x,y)=xy+H_{n+1}(x,y)$ has type $ \mathbf{A}_k$ at any isolated critical point. In particular, $H(x,y)$ has type $ \mathbf{A}_1$ at any isolated critical point on $L_{0}$.
	
\end{theorem}

It should be noted that there may be additional isolated critical points 
on $L_0$, specifically, point $O_j$ that is distinct from the origin $O$. In such a case,  it is still non-degenerate by Theorem \ref{th-sim-s}. Denote by $ \gamma_{h,j}$ the   vanishing cycle associated to  $O_j$  with the same orientation to $\gamma_h$ in which $T(h)$ has a form $2\pi \mathrm{i}+o(1)$, where  $o(1)$ represents the infinitesimal part as $h\rightarrow 0$.
Furthermore, the corresponding period function can be defined as the integral of the form  $\oint_{\gamma_{h,j}}\omega$.  It  possesses the following property.

\begin{theorem}\label{th-l0-p}

For system \eqref{m},
 the period function $\oint_{\gamma_{h,j}}\omega$ has a form $ \frac{2\pi \mathrm{i}}{n-1}+o(1) $ at any  isolated singular point  $O_j(\neq O)\in L_0$, in a neighborhood of $h=0$.	
\end{theorem}

The information provided by the dynamics at finite singular points is insufficient for determining whether the origin is isochronous. 
To gain a complete understanding, it is necessary to examine the behaviour of $L_h$ at infinity. In order to address this issue, we present the following three theorems which describe the properties at  infinity for system \eqref{m} that does not satisfy any one of the conditions in Theorem \ref{th-main}. Denote by  $\overline{L}_h$ the projective closure of a fiber $L_h$ in $\mathbb{CP}^2$.

\begin{theorem}\label{th-non-pole1}
For the system \eqref{m} with polynomial $H(x,y)=xy+H_{n+1}(x,y)$,	if  $a_{j}a_{k}\neq 0$ for two numbers $j,k$ such that $0\leq j< {(n+1)}/{2}<k\leq n+1 $, 
then
 the period $1$-form $\omega$ restricting on $\overline{L}_h$  does not have any pole at infinity  for any $h\neq 0$.	
\end{theorem}

In \cite{Gavri} the topological index  $\lambda^{h}(H)$ was introduced in two equivalent ways, one of which is outlined below.
Denote by $\mathcal{D}(h,x)$ the discriminant of $H(x,y)-h$ with respect to $y$. Let $deg(h,x)$ be
the degree of $\mathcal{D}(h, x)$ in $x$, and let $deg(x)$ be the degree of $\mathcal{D}(h, x)$ for generic $h$. Then the topological index $\lambda^{h}(H) $ is defined as $ deg(x)-deg(h,x)$. This number is equal to the number of ramification points  of the map $\tau_{h'}$ which tend to infinity as $h'$ tends to $h$, where the map $\tau_{h'}:\ L_{h'}\rightarrow \mathbb{C}$ is defined by $\tau_{h'}(x,y)=x$.
Let  $\lambda_{\mathrm{P}}^{h}(H)$ be the number of ramification points on $L_{h'}$ tending to a point $\mathrm{P}\in \overline{L}_h-L_h$ at infinity as $h'$ tends to $h$. It is evident that  $\lambda^{h}(H)$ can be expressed as the sum of the contributions from each point $\mathrm{P}\in \overline{L}_h-L_h $: $$\lambda^{h}(H)=\sum_{\mathrm{P}\in \overline{L}_h-L_h} \lambda_{\mathrm{P}}^{h}(H).$$

The projective coordinate of $ \mathrm{P}$ in $\mathbb{CP}^2$ is denoted by $[\beta:\alpha:0]$  (see the next section for the definitions of the projective coordinate and the  multiplicity of $\mathrm{P}$). We have the following theorem.

\begin{theorem}\label{th-lambda}
For polynomial $H(x,y)=xy+H_{n+1}(x,y)$, assume $a_{j}a_{k}\neq 0$ for two numbers $j,k$ such that $0\leq j< {(n+1)}/{2}<k\leq n+1 $ and   $\mathrm{P} \in \overline{ L}_h$ is a point at infinity.  Then 
 $ \lambda^{h}_{\mathrm{P}}(H)\neq 0$ if and only if $h=0$ and $\mathrm{P}=[1:0:0]$ or $[0:1:0]$ with multiplicity $N$ such that $1<N<(n+1)/2$.
\end{theorem}

Denote by  $\mathrm{P}_x=[1:0:0]$ and $\mathrm{P}_y=[0:1:0]$.
Under the assumption of Theorem \ref{th-lambda}, it can be stated that $ \lambda^{h}_{\mathrm{P}}(H)= 0$ for any $h\neq0$ and any $\mathrm{P}\neq \mathrm{P}_x,\ \mathrm{P}_y$. Even though on $L_0$, 
only  $ \lambda^{0}_{\mathrm{P}_x}(H)$ and $ \lambda^{0}_{\mathrm{P}_y}(H)$   may be non-zero, while for any other $\mathrm{P}\in  \overline{L}_0-L_0$, $ \lambda^{0}_{\mathrm{P}}(H)= 0$.

  In this paper,  a cycle in  $ \mathcal{H}_1(L_{h'},\mathbb{Z})$ vanishing at an infinite point $\mathrm{P}$ as $h'\rightarrow h$ means that it vanishes at $\mathrm{P}$ when treated as a cycle on the compact Riemann surface of $\overline{L}_{h'}$ (i.e., a resolution of $\overline{L}_{h'}$) as $h'\rightarrow h$. The equation $\lambda^{h}_{\mathrm{P}}(H)= 0$ implies that there are no  non-trivial cycles in  $ \mathcal{H}_1(L_{h'},\mathbb{Z})$ vanishing at $\mathrm{P}$ as $h'\rightarrow h$,  which are still non-trivial on the compact Riemann surface of $\overline{L}_{h'}$.
  In cases where either  $ \lambda^{0}_{\mathrm{P}_x}(H)\neq 0$ or $ \lambda^{0}_{\mathrm{P}_y}(H)\neq 0$, the polynomial  
 $H$ can be expressed as  $H(x,y)=xy+y^{N}\tilde{H}_{n+1-N}(x,y)$ or $H(x,y)=xy+x^{N}\tilde{H}_{n+1-N}(x,y)$   respectively, where $1<N<(n+1)/2$.  It follows that 
there exist  cycles in  $\mathcal{H}_1({L}_h,\mathbb{Z})$ vanishing at $\mathrm{P}_{x}$ or $\mathrm{P}_{y}$ as $h\rightarrow 0$. These cycles  can be divided  into two classes according to whether they are trivial or not on the Riemann surface of $\overline{L}_h$. Let $G_{\mathrm{P}_{x}}$ ( resp. $G_{\mathrm{P}_{y}}$) be the subgroup of  $\mathcal{H}_1(L_h,\mathbb{Z})$ generated by the cycles that vanish at $\mathrm{P}_{x}$ (resp.  $\mathrm{P}_{y}$) as $h\rightarrow 0$ and  remain non-trivial on the  Riemann surface of $\overline{L}_h$. 
We have   the following result.

\begin{theorem}\label{th-vc-inf-0}
For system 	(\ref{m}),	assume $a_{j}a_{k}\neq 0$ for two numbers $j,k$ such that $0\leq j< {(n+1)}/{2}<k\leq n+1 $.
Then the following statements hold.
\begin{enumerate}
\item[$(\mathrm{I})$] 
If  $H_{n+1}(x,y)=y^{N_1}\tilde{H}(x,y)$ with $\tilde{H}(1,0)\neq 0$ and $1<N_1 <(n+1)/2$, then there are $N_1$ different non-trivial cycles in $ \mathcal{H}_1(L_h,\mathbb{Z})$ generating  $G_{\mathrm{P}_{x}}$. Among those cycles, there is one cycle having  the period of the form  ${2\pi\mathrm{i}}+o(1)$, and for each of the remaining cycles the period  has  the form  $\frac{2\pi\mathrm{i}}{
N_1-1}+o(1)$, in the same orientation to $\gamma_h$.	

  


  \item[$(\mathrm{II})$]  
If  $H_{n+1}(x,y)=x^{N_2}\tilde{H}(x,y)$ with $\tilde{H}(0,1)\neq 0$ and $1<N_2 <(n+1)/2$, then there are $N_2$ different non-trivial cycles in $ \mathcal{H}_1({L}_h,\mathbb{Z})$ generating  $G_{\mathrm{P}_{y}}$. Among those cycles, there is one  cycle having  the period of the form  ${2\pi\mathrm{i}}+o(1)$, and for each of the remaining cycles the period  has  the form  $\frac{2\pi\mathrm{i}}{
N_2-1}+o(1)$, in the same orientation to $\gamma_h$.	

\end{enumerate}

\end{theorem}

 The structure of the paper is organized as follows. In Section 2  we provide some brief preliminaries. In Section 3, we shall prove Theorem \ref{th-sim-s} to Theorem \ref{th-vc-inf-0} and  introduce  additional lemmas on non-isolated singular points and  the intersection form between the vanishing cycles.
 In Section 4, 
 we will present a detailed proof of Theorem \ref{th-main}.

\section{Points at infinity}

For system (\ref{HSH}), if the vanishing cycle associated to a center is homologous to the zero cycle on the Riemann surface of the projective closure $\overline{L}_h$ of a generic level curve $L_h$, then the period $T(h)$ can be calculated 
	 by the residues of the period $1$-form $\omega  $ at its poles on $\overline{L}_h$. It should be noted that all of these poles lie at infinity. In order  to study the infinite points, it is preferable to embed the fibers $L_h$ into the projective space $\mathbb{CP}^2$. In this case,
the projective closure $\overline{L}_h$  can be defined by the following homogeneous equations
$$\sum_{k=2}^{n+1}z^{n+1-k}H_{k}(x,y) -h z^{n+1}=0,\  [x:y:z]\in \mathbb{CP}^2.$$

Rewrite  the homogeneous part $H_{n+1} (x,y)$ as follows:
\begin{equation*}
	H_{n+1} (x,y)= \prod_{j=1}^{m} \left(\alpha_j x - \beta_j y\right)^{n_j},
	\ \ n_j \geq 1, \ \ \sum_{j=1}^{m}n_j =n+1,
\end{equation*}
where $\alpha_j,\ \beta_j \in \mathbb{C}$ such that $\alpha_k :\beta_k\neq \alpha_j :\beta_j$ if $k\neq j.$
The projective coordinate of a point $\mathrm{P}^{j}$ at infinity on $L_{h}$ can be
represented by $[\beta_j :\alpha_j : 0]$ and $n_j$ is said to be its {\it multiplicity}.
Up to a projective change of coordinates,
the point $[\beta_j :\alpha_j : 0]$ can be changed to $[1:0:0]$.
It is therefore convenient to adopt a pair of new affine coordinates, namely  $$(X,Y)=\left(\frac{1}{x},\frac{y}{x}\right),$$
and the Puiseux series near $\mathrm{P}^{j}$ are  fully determined by
the ones for  the equation $H_{h}^{\ast}(X,Y)=0$ near the point $(X,Y)=(0,0)$, where
\begin{equation*}
	H_{h}^{\ast}(X,Y):=X^{n+1} H\left(\frac{1}{X},\frac{Y}{X}\right)-hX^{n+1}.
\end{equation*}

According to the classical  Puiseux's theory,
each branch of an algebraic curve near a singular point can be  parameterized  by a Puiseux series of the following form (see, e.g., \cite{Fischer}).
\begin{lemma}[Puiseux]\label{Pui-par}
	If $H_{h}^{\ast}(0,0)=0$ and $H_{h}^{\ast}(0,Y)\neq 0$,
	then there exist  numbers $\mathsf{p}, \mathsf{q}\in \mathbb{Z}_{+}$,
	a parameter $s\in \mathbb{C}$, and a  holomorphic function
$$\rho(s)=s^{ \mathsf{q}} \left(d_0+\sum_{k= 1}^{+\infty}d_k s^k\right),\quad  d_0 \neq 0,$$
	such that $H_{h}^{\ast}(s^{\mathsf{p}}, s^{ \mathsf{q}} \rho(s))= 0$ for all $s$ in a neighborhood of $0$.
\end{lemma}

By employing the Puiseux parameterization $x=s^{-\mathsf{p}}, \ y=s^{\mathsf{q}-\mathsf{p}}\rho(s)$ in the period $1$-form $\omega$,
we obtain the following expression
\begin{equation}\label{vf-inf}
	\omega = \frac{ \mathsf{p} s^{-\mathsf{p}-1} \ ds}{ H_y(s^{-\mathsf{p}},s^{\mathsf{q}-\mathsf{p}}\rho(s))}
\end{equation}
on one of branches of $\overline{L}_h$ in a neighborhood of $\mathrm{P}^j$.
This relation allows us to calculate the residues of the 1-form $\omega$ at infinite poles.

The Puiseux parameterization can be completely determined  by the  {\it  Newton polygon} of the singular point.
Given an irreducible polynomial $$F(X,Y)=\sum_{k,l}b_{kl}X^k Y^l,\ \ 
F(0, 0) = 0,$$ we may define  the carrier $\Lambda(F)$ of $F(X,Y)$ as
$\Lambda(F) = \{(k, l) \in \mathbb{Z}^2 \ \vert \  b_{kl}\neq 0\}.$
Assume that $R_1,R_2\in \mathbb{R}^2$ and let 
$$[R_1,R_2]=\{\tau R_1 +(1-\tau)R_2\ \vert \ 0\leq \tau\leq 1\}$$
be the straight line segment from $R_1$ to $R_2$. Furthermore,  
consider the convex subset $E$ on $\mathbb{R}^2$ consisting of those $(X,Y)\in \mathbb{R}^2$ such that $X\geq X_0$ and $Y\geq Y_0$  for some $(X_0,Y_0)\in [R_1,R_2]$ where $R_1,R_2\in \Lambda(F)$.
\begin{definition}[Newton Polygon]

The boundary of set $E$, excluding the axes, is called the {\it Newton polygon}  of $F(X,Y)$ at the origin, which consists  of a finite number of straight line segments.
\end{definition}

\section{Topological properties of the polynomial Hamiltonian function}

This section is dedicated to the proof of the theorems that describe the topological properties of system \eqref{m}, namely, Theorem \ref{th-sim-s} to Theorem \ref{th-vc-inf-0}.
Prior to this, two simple lemmas are introduced.  

\begin{lemma}\label{lem-noiso-n}
For system 	(\ref{m}), if $n$ is odd and $a_{k}\neq 0$ for $k=(n+1)/2$, 
 then the origin cannot be isochronous.
\end{lemma}
\begin{proof} 

Note that, for $k=(n+1)/2$, $k a_k$ is precisely the coefficient of the first resonant term in system \eqref{m}. Consequently, if this coefficient is non-zero, then the linearization procedure  is unable to eliminate this resonant term, which  in turn implies that the origin cannot be isochronous.
\end{proof}

It is possible that  the system \eqref{m} may exhibit  non-isolated singular points. Nevertheless, in such a case, it must  possess a highly specialised form.

\begin{lemma}\label{lem-nonisol}
The system 	(\ref{m}) exhibits non-isolated singular points if and only if $n$ is odd and $H_{n+1}(x,y)=a_{N}x^N y^N$, where $a_N \neq 0$ and $N=(n+1)/2$.
\end{lemma}

\begin{proof}
Clearly system 	(\ref{m}) has non-isolated singular points if and only if there exist non-constant polynomials $P(x,y)$, $Q(x,y)$ and $f(x,y)$ of degree $m_1,$ $m_2$ and $m_3$ respectively, such that 
\begin{equation}\label{eq-nonisol}
H_x(x,y)=P(x,y)f(x,y),\ \ H_y(x,y)=Q(x,y)f(x,y).	
\end{equation}
Without loss of generality, we can assume that
$$P(x,y)=y+\sum_{j\geq2} P_j(x,y), \ Q(x,y)=x+\sum_{j\geq 2} Q_j(x,y), \  f(x,y)=1+\sum_{j\geq 1} f_j(x,y),$$
 where $P_j, Q_j
$ and $ f_j$ are homogeneous polynomials of degree $j$. By comparing the coefficients of  the two sides of the equations (\ref{eq-nonisol}), one can get the following equalities inductively:
\begin{equation*}
\begin{array}{llll}
P_2 =-y f_1, & P_3=y(f_1^2 -f_2),&...,& P_{m_1}=y((-f_1)^{m_1-1}+...-f_{m_1 -1}),\\ 
Q_2 =-x f_1, & Q_3=x(f_1^2 -f_2),&...,& Q_{m_2}=x((-f_1)^{m_1-1}+...-f_{m_1 -1}).
\end{array}
\end{equation*}
That is, $xP_j=yQ_j$ for any $j\leq \min\{m_1,m_2\}$.
For $m_1 \neq m_2$, it follows that $H_{n+1}$ is one of two possibilities:  either  $H_{n+1}=a_0 x^{n+1}$ or  $H_{n+1}=a_{n+1} y^{n+1}$. This, in turn, implies that $f\equiv 1$, which conflicts to the assumption that $f$ is not a constant.
In the case where $m_1 = m_2$, we have 
$xP_{m_1}f_{m_3}-yQ_{m_2}f_{m_3}=0$ and the Euler's identity $$xP_{m_1}f_{m_3}+yQ_{m_2}f_{m_3}=x\frac{\partial H_{n+1}}{\partial x} +y\frac{\partial H_{n+1}}{\partial y}=(n+1) H_{n+1}. $$
These yield the following equalities:
$$x\frac{\partial H_{n+1}}{\partial x} =y\frac{\partial H_{n+1}}{\partial y}=\frac{(n+1) H_{n+1}}{2}.$$
The above equation holds if and only if  the following two conditions are met: firstly, that $n$ is an odd number; and secondly, that $H_{n+1}=a_{N}x^N y^N$, where $a_N \neq 0$ and $N=(n+1)/2$. 
\end{proof}

The following corollary can be obtained directly from Lemma \ref{lem-noiso-n}  and Lemma \ref{lem-nonisol}.
\begin{corollary}\label{cor-noniso}
	If system \eqref{m} has non-isolated singular points, then the origin is not isochronous.
\end{corollary}

Now we shall prove Theorem \ref{th-sim-s}.

\begin{proof}[Proof of Theorem \ref{th-sim-s}]

For the sake of convenience in the notation of partial derivatives, we shall  replace $H_{n+1}$ with $Q$ in this proof.
	   Assume that $O_0=(x_0,y_0)$ is an isolated critical point of $H(x,y)$ on $L_{h_0}$, where $h_0=H(x_0,y_0)$. Clearly $x_0$ and $y_0 $ satisfy the following equations simultaneously:
\begin{equation}\label{eq-sing}
	\begin{array}{c}
	x_0 y_0 +Q(x_0,y_0)=h_0,\ \ x_0 + Q_y(x_0,y_0)=0, \ 	y_0 + Q_x(x_0,y_0)=0.
	\end{array}
\end{equation}
In addition, the following Euler's identities hold:
	\begin{equation}\label{eq-euler}
	\begin{array}{c}
	xQ_x +yQ_y =(n+1)Q,\ \ xQ_{xx}+yQ_{xy}=nQ_{x},\ \ xQ_{yx}+yQ_{yy}=nQ_{y}.
	\end{array}
\end{equation}

In the case of $h_0=0$, equations  (\ref{eq-sing}) and (\ref{eq-euler}) yield that  $x_0 y_0 =0$ and $Q(x_0,y_0)=0$.	The origin is non-degenerate obviously. If $x_0\neq 0$ but $y_0= 0$, then $a_{0}=0,\ a_1\neq 0$ and $x_0$ is a root of the equation $1+ a_1 x^{n-1}=0$. At the point $(x_0,0)$, 
the Hessian matrix   of $H(x,y)$  is given by
\begin{equation*}
\begin{array}{lcl}
	\left(\begin{array}{cc}
 H_{xx} & H_{xy} \\
 H_{yx} & H_{yy}  
\end{array}\right)&=& \left(\begin{array}{cc}
 0 & 1+na_{1}x_0^{n-1}  \\
 1+na_{1}x_0^{n-1} &  2a_{2}x_0^{n-1}  
\end{array}\right)\\&=&\left(\begin{array}{cc}
 0 & -(n-1) \\
 -(n-1) & 2a_{2}x_0^{n-1}  
\end{array}\right),
\end{array}
\end{equation*}
which is non-degenerate. Thus $(x_0,0)$ is a non-degenerate critical point. The similar arguments can be also applied to the case where $x_0= 0$ but $y_0\neq 0$.

 In the  case of $h_0\neq 0$, it is sufficient to demonstrate that the Hessian matrix
	of $H(x,y)$ is not identically zero  at $O_0$, i.e. its rank $\mathsf{r} \geq 1$, which implies that $H(x,y)$  is non-degenerate at $O_0$ for  $\mathsf{r}=2$ and has  type $\mathbf{A}_k,\ k>1$ for  $\mathsf{r}=1$.

	If $H_{xy}(x_0,y_0)=1+Q_{xy}(x_0,y_0)\neq 0$, then the proof is complete. In the case where $H_{xy}(x_0,y_0)=1+Q_{xy}(x_0,y_0)= 0$,  we will show that $$Q_{xx}(x_0,y_0)Q_{yy}(x_0,y_0)\neq 0,$$ which also finishes the proof.

Firstly, noticing that $h_0\neq 0$,  the  equations  (\ref{eq-sing}) yield that $$x_0y_0 ={(n+1)h_0}/{(n-1)}\neq 0.$$ 	
Secondly, by equations (\ref{eq-sing}) and (\ref{eq-euler}), we obtain the following results: 
$$ Q_{xx}(x_0,y_0)=\frac{nQ_{x}(x_0,y_0)+y_0}{x_0}=-\frac{(n-1)y_0}{x_0}\neq 0,$$
and 	
$$ Q_{yy}(x_0,y_0)=\frac{nQ_{y}(x_0,y_0)+x_0}{y_0}=-\frac{(n-1)x_0}{y_0}\neq 0.$$
\end{proof}

The proof of the above lemma enables us to derive the following corollary.

\begin{corollary}\label{cor-l0}
For polynomial $H(x,y)=xy+H_{n+1}(x,y)$,	the  level curve $L_0$ contains two or more critical points if and only if one of the following conditions holds:
	\begin{enumerate}
  \item[$\mathrm{(I)}$] $a_0= 0$ but $a_1\neq 0$;
  \item[$\mathrm{(II)}$] $  a_{n+1}=0$ but $a_n \neq 0$.\end{enumerate}

\end{corollary} 	

Next we will prove Theorem \ref{th-l0-p}.

\begin{proof}[Proof of Theorem \ref{th-l0-p}]

By Corollary \ref{cor-l0}, it is sufficient to prove that the theorem under the condition  $a_0= 0, a_1\neq 0$ or the condition $  a_{n+1}=0, a_n \neq 0$.

In the case where $a_0= 0$ but $a_1\neq 0$, it follows that for each root $x_j$ of the equation $1+a_1 x^{n-1}=0$,  $(x_j,0)$  is a singular point on $L_{0}$ which is not the origin. By Theorem \ref{th-sim-s}, the  point $(x_j,0)$ is still non-degenerate. Denote by $ \gamma_{h,j0}$ the vanishing cycle associated to  $(x_j,0)$  with the same orientation to $\gamma_h$ in which $T(h)$ has a form $2\pi \mathrm{i}+o(1)$.

In a neighborhood of $(x_j,0)$,	
	the system  (\ref{m}) can be expressed as 
	\begin{equation*}
		\frac{dx}{dt}= -(n-1)x-\frac{2a_2}{a_1}y+h.o.t,\ \ \frac{dy}{dt}= (n-1)y+h.o.t.
			\end{equation*}	
Consequently, by observing the orientation of 		 $\gamma_{h,j0}$, the  period function $$T_{j0}(h):=\oint_{\gamma_{h,j0}}\omega =\oint_{\gamma_{h,j0}}\frac{dy}{(n-1)y+h.o.t}=\frac{2\pi \mathrm{i}}{n-1}+o(1)$$
in a neighborhood of $h=0$.

  In the case where $  a_{n+1}=0$ but $a_n \neq 0$, let $y_j$ be a root of the equation $1+a_n y^{n-1}=0$.  
  Similarly to the  former  case,  we can obtain that  $T_{0j}(h):=\oint_{\gamma_{h,0j}}\omega=\frac{2\pi \mathrm{i}}{n-1}+o(1)$
for the vanishing cycle $  \gamma_{h,0j}$ associated to the singular point $(0,y_j)$.
\end{proof}

The rest of this section is devoted to  the proof of   Theorems \ref{th-non-pole1}, \ref{th-lambda} and \ref{th-vc-inf-0}.

\begin{proof}[Proof of Theorem \ref{th-non-pole1}]

If one of the points at infinity has a projective coordinate $[1:0:0]$, i.e.,
  $H_{n+1}=y^{N}\tilde{H}_{n+1-N}$ has a linear factor $ y$, then  its  multiplicity $N<(n+1)/2$ according to the assumptions of the theorem. 
  
  In the affine chart $(X,Y)=(1/x,y/x)$, the point $[1:0:0]$ is located at the origin. 
  For any $h\neq 0$, 
the corresponding 
Newton polygon of $H^*_{h}(X,Y)=0$ at the origin has  only one straight line segment with two endpoints $(k_1,l_1)=(0,N)$ and $(k_2,l_2)=(n+1,0)$, see the left picture in Figure \ref{fig:NP}. Thus, the Puiseux parameterization of any branch of the origin  can be expressed as follows:
$$ X=s^{N}, \ \ Y=s^{n+1}(d_0+h.o.t),$$
which leads to a local expression of the $1$-form $\omega$ restricted to this branch:
 $$\omega =\frac{ s^{n-2N}ds}{d^{N-1}_0\tilde{H}_{n+1-N}(1,0)+h.o.t}.$$
Due to the facts that $d_0 \tilde{H}_{n+1-N}(1,0)\neq 0$ and $2N\leq n$, the $1$-form  $\omega$ does not possess any pole on any branch of the point $[0:1:0]$. The above arguments are also valid for the point $[0:1:0]$.

If $H_{n+1}$ has a linear factor $\alpha x-\beta  y$ with multiplicity $N$ and $\alpha\beta \neq 0$,
then  a change of coordinates $(\tilde{ x},\tilde{y})=(x, \alpha x -\beta y)$   transforms the point  $[\beta : \alpha :0]$ to $[1:0:0]$ and brings the function $H(x,y)$  to the form $$
F(\tilde{ x},\tilde{y})=\frac{1}{\beta}\tilde{x}\left(\alpha\tilde{x}-\tilde{y}\right)+\tilde{y}^N f(\tilde{ x},\tilde{y}),$$
where $f(1,0)\neq 0$. In the affine chart $(X,Y)=(1/\tilde{ x},\tilde{ y}/\tilde{ x})$, the Newton polygon of the corresponding polynomial   
$$F^{\ast}_{h}(X,Y)=X^{n+1}F\left(\frac{1}{X},\frac{Y}{X}\right)-h X^{n+1}$$ 
  at the origin  also exhibits  a single straight line segment with two endpoints $(k_1,l_1)=(0,N)$ and $(k_2,l_2)=(n-1,0)$, see the right picture in Figure \ref{fig:NP}. Thus, the Puiseux parameterization of any branch has the following form
$$ X=s^{N}, \ \ Y=s^{n-1}(d_0+h.o.t),$$
which gives a local expression of the $1$-form $\omega$ restricted to this branch
as follows: $$\omega =\frac{ s^{n-2}ds}{d^{N-1}_0f(1,0)+h.o.t}.$$
Clearly $d_0 f(1,0)\neq 0$ and $n\geq 2$. Therefore, the $1$-form $\omega$ cannot have a pole on any branch of the point $[\beta : \alpha :0]$.	
 \begin{figure}[htbp]
\centering
	\includegraphics[]{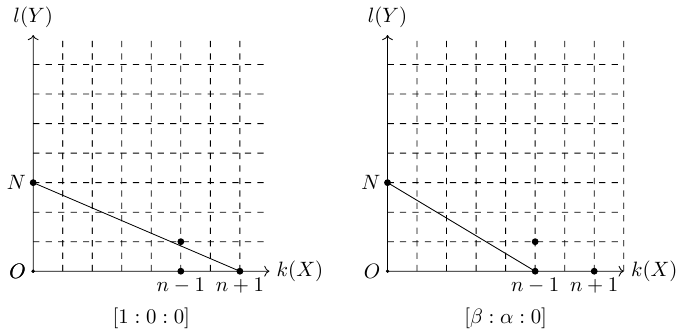}
	\caption{Newton polygon at an infinite point}\label{fig:NP}
\end{figure}   
\end{proof}

To prove Theorem \ref{th-lambda},   recall that $L_{h'}$ is a smooth affine curve defined by  $H(x,y)-h'=0$ for any generic value $h'\not \in A_{H} $. Consequently, a point $(x_0,y_0)\in L_{h'}$ is a ramification point of  $\tau_{h'}$ if and only if    $H_{y}(x_0,y_0)=0$.

\begin{proof}[Proof of Theorem \ref{th-lambda}]

Let $\mathrm{P}=[\beta:\alpha:0]$ be an infinite point with multiplicity $N$. Then
  $H_{n+1}$ has the form $(\alpha x-\beta y)^{N}\tilde{H}_{n+1-N}(x,y)$ with $\tilde{H}_{n+1-N}(\beta,\alpha)\neq 0$. 
  
  We first consider $\beta\neq 0$. Subsequently, we may set  $\beta=1$.    
  Suppose $(x_0,y_0)\in L_{h'}$ is a ramification point of the map $\tau_{h'}$, then we have  $H(x_0,y_0)-h'=0$ and $ H_{y}(x_0,y_0)=0$. If $(x_0, y_0)$ tends to $\mathrm{P}=[1:\alpha:0]$, then $x_0$ tends to infinity. We can therefore assume $x_0\neq 0$.
 Let $X_0=1/x_0$ and $Y_0=(\alpha x_0- y_0)/x_0$. Showing that $(x_0,y_0)$ tends to $\mathrm{P}$ is equivalent to showing that $(X_0,Y_0)$ tends to the origin $(0,0)$. 
 From the equation  $  H_{y}(x_0,y_0)=0$ we obtain the following equality:  
\begin{equation}\label{bif-Y}
	Y_0^{N-1}\left(-N\tilde{H}_{n+1-N} (1,\alpha-Y_0)+Y_0 \frac{\partial\tilde{H}_{n+1-N}}{\partial y}(1,\alpha-Y_0)\right)+X_0^{n-1}=0.
\end{equation}
 
 If $N=1$, then the above equality becomes
 \begin{equation*}
	\left(-\tilde{H}_{n} (1,\alpha-Y_0)+Y_0 \frac{\partial\tilde{H}_{n}}{\partial y}(1,\alpha-Y_0)\right)+X_0^{n-1}=0.
\end{equation*} 
Thus, the point $(X_0,Y_0)$ cannot tend to  $(0,0)$, given that $ \tilde{H}_{n} (1,\alpha)\neq 0$.

 When $N>1$,
 by taking \begin{equation*}
	X_0^{n-1}=Y_0^{N-1}\left(N\tilde{H}_{n+1-N} (1,\alpha-Y_0)-Y_0 \frac{\partial\tilde{H}_{n+1-N}}{\partial y}(1,\alpha-Y_0)\right)
\end{equation*}
into $H(x_0,y_0)-h'= H(1/X_0,(\alpha -Y_0)/X_0)-h'=0$, we get the following equality  
 \begin{eqnarray}\label{bif-lim2}
 \begin{array}{l}
 \left(\left(\alpha N -(N-1)Y_0\right)\tilde{H}_{n+1-N} (1,\alpha-Y_0)-(\alpha-Y_0)Y_0 \frac{\partial \tilde{H}_{n+1-N}}{\partial y}(1,\alpha-Y_0)\right)^{n-1}\\=h'^{(n-1)}Y_0^{2(N-1)}\left(N\tilde{H}_{n+1-N} (1,\alpha -Y_0)-Y_0 \frac{\partial\tilde{H}_{n+1-N}}{\partial y}(1,\alpha-Y_0)\right)^{n+1}.	
 \end{array}
 \end{eqnarray}
 For $\alpha\neq 0$, the point $(X_0,Y_0)$ cannot tend to $ (0,0)$ as $h'\rightarrow h$ for any $h$,  since $\tilde{H}_{n+1-N} (1,\alpha)\neq 0$.
 
    In the case of  $\alpha=0$, one can obtain that $N<(n+1)/2$ from the assumption that there exist two numbers $j<(n+1)/2<k$ such that $a_j a_k\neq 0$.
   Now   the equality (\ref{bif-lim2}) becomes 
  \begin{equation}\label{bif-lim3}
 \begin{array}{rcl}
 &&Y_0^{n+1-2N}\left(\left( -(N-1)\right)\tilde{H}_{n+1-N} (1,-Y_0)+Y_0 \frac{\partial \tilde{H}_{n+1-N}}{\partial y}(1,-Y_0)\right)^{n-1}\\&=&h'^{(n-1)}\left(N\tilde{H}_{n+1-N} (1, -Y_0)-Y_0 \frac{\partial\tilde{H}_{n+1-N}}{\partial y}(1,-Y_0)\right)^{n+1}.	
 \end{array}
 \end{equation} 
 For $h\neq 0$, it is not possible for $Y_0$ to approach $0$ as $h'\rightarrow h$.  If this is not the case, then by letting $h'\rightarrow h$ and $Y_0 \rightarrow0 $ in equation  \eqref{bif-lim2}, we obtain  $\tilde{H}_{n+1-N} \left(1,0\right)=0$, which leads to a contradiction. Therefore, $(x_0,y_0)$ cannot tend to $ \mathrm{P}=\mathrm{P}_{x}=[1:0:0]$ as $h'\rightarrow h$.
 As for $h=0$ and $N>1$, there  indeed exist  non-zero solutions $(X_0,Y_0)$ of equations \eqref{bif-Y} and   \eqref{bif-lim3} that tend to $0$ as $h'\rightarrow 0$. 
 
 So far, we have proven that for $\beta\neq 0$ the number $\lambda^{h}_{\mathrm{P}}(H)\neq 0$ if and only if  $h= 0$ and $\mathrm{P}=\mathrm{P}_{x}$ with multiplicity $N>1$.  
 The same arguments can be also applied to the case of $\beta=0$, i.e., 
  $\mathrm{P}=\mathrm{P}_{y}=[0:1:0]$. In this case, the multiplicity $N$ of $\mathrm{P}_{y}$ also satisfies $1\leq N <(n+1)/2$. We may now assume that $y_0\neq 0$, and let $ (X_0,Y_0)=(x_0/y_0,1/y_0)$. The equality $ H_y(x_0,y_0)=0$ is equivalent to 
\begin{equation*}
 \begin{array}{c}
 Y_0^{n-1}+X_0^{N-1}\frac{\partial\tilde{H}_{n+1-N}}{\partial y}(X_0,1)=0.	
 \end{array}
 \end{equation*}  
For $N=1$,  the above equality indicates that  $(X_0,Y_0)$ cannot tend to $(0,0)$, due to that   $\frac{\partial\tilde{H}_{n}}{\partial y}(0,1)=na_{n}\neq 0$.
For $1<N<(n+1)/2$ we take \begin{equation*}
	Y_0^{n-1}=-X_0^{N-1} \frac{\partial\tilde{H}_{n+1-N}}{\partial y}(X_0,1)
\end{equation*}
into $H(x_0,y_0)-h'= H(X_0/Y_0,1/Y_0)-h'=0$, and obtain the following equality  
 \begin{equation*}
 \begin{array}{lcl}
 &&X_0^{n+1-2N}\left(\frac{\partial \tilde{H}_{n+1-N}}{\partial y}(X_0,1)-\tilde{H}_{n+1-N} (X_0,1) \right)^{n-1}\\&=&h'^{(n-1)}\left( \frac{\partial\tilde{H}_{n+1-N}}{\partial y}(X_0,1)\right)^{n+1}.	
 \end{array}
 \end{equation*}
Noticing that $\frac{\partial\tilde{H}_{n+1-N}}{\partial y}(0,1)=(n+1-N)a_{n+1-N}\neq 0$, it is not difficult to see that there exist points $(X_0,Y_0)$ tending to $ (0,0)$ as $h'$ tends to $ h$ if and only if $h=0$.
The proof of the theorem is finished.
  \end{proof}

 Under the conditions of Theorem \ref{th-main}, one can obtain that  $ \lambda_{\mathrm{P}_x}^0(H)$  or $ \lambda_{\mathrm{P}_y}^0(H)$ is  also equal to $0$, by a parallel analysis to that employed in the proof of  Theorem \ref{th-lambda} with the additional observation that $2N> n+1$.  Namely, the following  lemma holds.

\begin{lemma}\label{lem-lambda}
For polynomial $H(x,y)=xy+H_{n+1}(x,y)$, if $a_j=0$  for any integer $j$ satisfying $0 \leq j \leq (n+1)/2 $
 $($resp. $a_j=0$ for any integer $j$ satisfying  $(n+1)/2 \leq j \leq n +1$$)$, then $ \lambda_{\mathrm{P}_x}^0(H)=0$  $($resp. $ \lambda_{\mathrm{P}_y}^0(H)=0$$)$.	
\end{lemma}

When $\lambda^{0}(H)\neq 0$,   $H(x,y)$ is not a good polynomial as defined in \cite{Gavri}.
In this case, it is necessary to determine  how the cycles vanishing at infinity affect the period  $T(h)$.	This is the aim of Theorem \ref{th-vc-inf-0}.

\begin{proof}[Proof of Theorem \ref{th-vc-inf-0}]

It is sufficient to prove the  statement $(\mathrm{I})$, given that the statement $(\mathrm{II})$  can be addressed in a similar way. 
In the affine chart $(X,Y)=(1/x,y/x)$, the point $\mathrm{P}_x$ is taken to be the origin.  The system \eqref{m} is then changed to 
\begin{equation} \label{sys-XY}
	\begin{array}{lcl}
		\frac{dX}{dt}&=&-\frac{1}{X^{n-2}}\left(X^{n-1}+Y^{N_1} (N_1 Q(1,Y)+Y Q_y (1,Y))\right),\\ 
		\frac{dY}{dt}&=&-\frac{1}{X^{n-1}}\left(2X^{n-1}Y+(n +1)Y^{N_1}  Q(1,Y))\right),
	\end{array}
\end{equation} 
in a neighborhood of  $\mathrm{P}_x$, where $Q(x,y)={H_{n+1}(x,y)}/{y^{N_1}}$	is a polynomial of degree $n+1-N_1$ with $Q(1,0)\neq 0$ and $Q_y= {\partial Q}/{\partial y}$.
The foliation induced by the vector field (\ref{sys-XY}) can be holomorphically extended to the origin, at which point it undergoes a degenerate singular point. In order to detect the topological structure of a generic leaf  near the origin, we use the method of quasi-homogenous blow-ups. That is, we perform a map 
\begin{equation*}
	\begin{array}{crcl}
	\Phi:&\mathbb{M} \subset \mathbb{C}^2\times \mathbb{CP}&\rightarrow  &	\mathbb{C}^2\\
	&(u,v)&\mapsto & (X,Y)=(u^{N_1 -1},u^{n-1}v),
	\end{array}
\end{equation*}
where $\mathbb{M}$ is a complex 2-dimensional manifold with coordinates $(u,v)$. 
Then system \eqref{sys-XY} is transformed to 
\begin{eqnarray} \label{sys-XY-blow}
	\begin{array}{l}
		\frac{du}{dt}=-\frac{u}{N_1-1}\left(1+N_1 v^{N_1-1} Q(1,u^{n-1}v)+u^{n-1}v^{N_1} Q_y (1,u^{n-1}v)\right),\\ 
		\frac{dv}{dt}=\frac{(n+1-2N_1)v}{N_1 -1}\left(1+v^{N_1-1}Q(1,u^{n-1}v)+\frac{n-1}{n +1-2N_1}u^{n-1} v^{N_1} Q_y(1,u^{n-1}v)\right).
	\end{array}
\end{eqnarray} 

It is not difficult to see that on the exceptional divisor $u=0$ there are exactly $N_1$ singular points $(0,v_0)=(0,0)$ and $\{(0,v_j),\ j=1,...,N_1-1\}$, where $\{v_j,j=1,...,N_1-1\}$ are $N_1 -1$ roots of the equation $1+a_{N_1} v^{N_1-1}=0$ with $a_{N_1}=Q(1,0)\neq 0$. 
More precisely, the origin is a resonant saddle with two eigenvalues $-\frac{1}{N_1-1}$ and $\frac{n+1-2N_1}{N_1-1}$. Similarly, the point $(0,v_j)$ is also a resonant saddle with two eigenvalues $1$ and $-(n+1-2N_1)$  by direct calculation. 

In a neighborhood  of the singular point $(0,v_0)$, the foliation induced by vector field \eqref{sys-XY-blow} is homeomorphic to the one induced by a linear vector field $$\frac{du}{dt}=-\frac{u}{N_1-1},\ \ \frac{dv}{dt}=\frac{(n+1-2N_1)v}{N_1 -1}, $$ 
which possesses a first integral $ u^{n+1-2N_1}v$.  
In addition, system \eqref{sys-XY-blow} has also a polynomial first integral
$$H(u^{-(N_1-1)},u^{n-N_1}v)=u^{n+1-2N_1}v(1+v^{N_1-1}Q(1,u^{n+1-2N_1}v)).$$ 
Thus
the fiber bundle defined by $H(u^{-(N_1-1)},u^{n-N_1}v)$ is locally homeomorphic to the one  defined by monomial $ u^{n+1-2N_1}v$. Note that a generic fiber of $ u^{n+1-2N_1}v$  is  a topological cylinder, it follows that   
 on a generic fiber of $H(u^{-(N_1-1)},u^{n-N_1}v)$, there is exactly one vanishing cycle, denoted by $\vartheta^*_{h, 0}$,  associated to the point $(0,v_0)$ generating the set of cycles vanishing at  $(0,v_0)$.
Similarly, since each point $(0,v_j)$ is an integrable 
resonant saddle with a polynomial first integral and two eigenvalues $1$ and $-(n+1-2N_1)$, a generic  fiber  near $(0,v_j)$ is also locally homeomorphic to a generic fiber of   $ u^{n+1-2N_1}v$, and  there is also exactly one vanishing cycle, denoted by $\vartheta^*_{h, j}$, associated to the point $(0,v_j)$ generating the set of cycles vanishing at $(0,v_j)$.

Let $ \vartheta_{h, j}=\Phi_{*}(\vartheta^{*}_{h, j})$ for  $j=0,1,...,N_1-1$, where $\Phi_{*}$ is a map between the first homology groups induced by $\Phi$ restricted on $\overline{L}_h-\mathrm{P}_x$. Namely, $\Phi_{*}$ maps a cycle represented by a closed loop $l$ to a cycle represented by the image $\Phi(l)$.Then $\vartheta_{h, j}\in G_{\mathrm{P}_x} $ and $\{\vartheta_{h, j}\}$ is a subset of a generating set of $G_{\mathrm{P}_{x}}$.  To  prove that $G_{\mathrm{P}_{x}}$ is exactly generated by $\{\vartheta_{h, j}\}$, we need to show that there are no other non-trivial cycles vanishing at $\mathrm{P}_{x}$. This is equivalent to show that there are no non-trivial cycles vanishing at the infinite point on the exceptional divisor $u=0$. To this end, one can take another quasi-homogenous blow-up $(X,Y)=(\tilde{u}\tilde{v}^{N_1 -1}, \tilde{v}^{n-1})$. Consequently the infinite point on $u=0$ is changed to the origin of the  $(\tilde{u},\tilde{v})$ chart,  and the system \eqref{sys-XY} is transformed to 
\begin{equation*}
	\begin{array}{lcl}
		\frac{d\tilde{u}}{dt}&=&-\frac{n+1-2N_1}{n-1}\tilde{u}\left(Q(1,\tilde{v}^{n-1})+\tilde{u}^{n-1}+\frac{n-1}{n+1-2N_1}\tilde{v}^{n-1} Q_y(1,\tilde{v}^{n-1}) \right), \\
		\frac{d\tilde{v}}{dt}&=&-\frac{n+1}{n-1}\tilde{v}\left(Q(1,\tilde{v}^{n-1})+\frac{2}{n+1}\tilde{u}^{n-1}\right).
	\end{array}
\end{equation*} 
Note that  the above system has a non-degenerate node at the origin. It follows that  each of the corresponding vanishing cycles can be represented by a small loop encircling the origin, which is trivial on the Riemann surface of $\overline{L}_h$. Therefore, $G_{\mathrm{P}_{x}}$ is generated by $\{\vartheta_{h, j}\}$.

It is not difficult to see that the loop representing the cycle $\vartheta_{h,0}$ tends to the branch $y=0$ (i.e., $Y=0$) of $\mathrm{P}_x$ on $L_0$ as $h\rightarrow 0$, whereas the loop representing $\vartheta_{h,j}$ tends to the other one branch  $Y^{N_1-1}-v_j X^{n-1}+h.o.t=0$ of $\mathrm{P}_x$ on $L_0$ for $j\neq 0$. Since there is only one finite singular point on $y=0$, we can obtain that
$ \vartheta_{h,0}$ is homologous to $\gamma_h$ on the Riemann surface of $\overline{L}_h$. That is, the difference between $ \vartheta_{h,0}$  and  $\gamma_h$  in $\mathcal{H}_1(L_h,\mathbb{Z})$ is a cycle represented by a closed loop encircling only one  point $\mathrm{P}_x$ at infinity.

  It should be noted that $\Phi $   is a $(N_1-1)$-to-$1$ orientation-preserving map out of the exceptional divisor $\Phi^{-1}(0)$. 
   Consequently, the period associated to $\vartheta_{h, jk}$ is equal to the period associated to $\vartheta^{*}_{h, jk}$, multiplied  by $1/(N_1-1)$. 
A direct calculation from system \eqref{sys-XY-blow}  indicates that 
 the period associated to $\vartheta^*_{h, 0}$  is equal to $2(N_1 -1)\pi \mathrm{i}+o(1)$ in the same orientation to $\gamma_h$.  Besides,  one can also get that the period associated to $\vartheta^*_{h, j}$ has a form  ${2\pi \mathrm{i}}+o(1)$. 
 Finally, the period of $\vartheta_{h, 0}$ takes  the form $2\pi \mathrm{i}+o(1)$ and the period of $\vartheta_{h, j}$ takes  the form  ${2\pi \mathrm{i}}/{(N_1-1)}+o(1)$.  
The proof of the theorem is finished.
 \end{proof}

Below we study the  variation operator $ \mathcal{V}_{{\varrho}x}$ that is defined on the  relative homology group, denoted by $\mathcal{H}_{1}(U_{x\epsilon},\partial U_{x\epsilon})$, in a small ball $B_{x\epsilon}$ centered at   $\mathrm{P}_x$ in chart $(X,Y)$ induced by a simple counterclockwise loop $ \varrho$ encircling only one critical value $h=0$, where $U_{x\epsilon}=B_{x\epsilon}\cap\overline{L}_h$ and $\partial U_{x\epsilon}$ is the boundary of $ U_{x\epsilon}$. For this,   we consider the corresponding 
variation operator $ \mathcal{V}^*_{{\varrho}x}$ on $\mathcal{H}_{1}(\Phi^{-1}(U_{x\epsilon}),\partial \Phi^{-1}(U_{x\epsilon}))$ induced by  $\varrho$ under the blow-up $\Phi$.

Denote by $\delta \cdot \gamma$ the intersection number between any two cycles  $\delta $ and $ \gamma$. Given a cycle $\delta_h\in \mathcal{H}_{1}(U_{x\epsilon},\partial U_{x\epsilon})$, it can be shown that $ \delta^{*}_h\cdot \vartheta^{*}_{h,j} =(N_1 -1)(\delta_h\cdot \vartheta_{h,j}), $ where $\delta^{*}_h= \Phi_{*}^{-1}(\delta_h)$.
Then, due to the property of the fiber bundle defined by $u^{n+1-2N_1}v$ and the fact that $ \vartheta^*_{h,j}\cdot \vartheta^*_{h,k}=0$ for any $j\neq k$, $j,k=0,1,...,N_1-1$, one can obtain  
$$\mathcal{V}^*_{{\varrho}x}(\delta^{*}_{h}) =-\sum_{j=0}^{N_1-1}(\delta^{*}_{h}\cdot \vartheta^*_{h,j}) \vartheta^*_{h,j}, $$ applying the Picard-Lefschetz formula for each point $(0,v_j)$.
This implies 
\begin{equation}\label{eq-var-px}
	\mathcal{V}_{{\varrho}x}(\delta_{h}) =-\sum_{j=0}^{N_1-1}(\delta_{h}\cdot \vartheta_{h,j}) \vartheta_{h,j}.
\end{equation}
If $\delta_{h}\in \mathcal{H}_{1}(L_h,\mathbb{Z})$ is an absolute cycle, then we can define $ \mathcal{V}_{{\varrho}x}(\delta_{h}):=\mathcal{V}_{{\varrho}x}(\delta_{h}\vert_{B})$, where $\delta_{h}\vert_{B}$ is a relative cycle in $ \mathcal{H}_{1}(U_{x\epsilon},\partial U_{x\epsilon})$ represented by $ \delta_h\cap B_{x\epsilon}$. Thus, the equality \eqref{eq-var-px} also holds for $\delta_{h}\in \mathcal{H}_{1}(L_h,\mathbb{Z})$.

For the point $\mathrm{P}_y$, a similar formula can be obtained  for the variation operator $\mathcal{V}_{{\varrho}y}$ on
the  relative homology group $\mathcal{H}_{1}(U_{y\epsilon},\partial U_{y\epsilon})$ in a small ball $B_{y\epsilon}$ centered at   $\mathrm{P}_y$ in chart $(X,Y)=(x/y,1/y)$ induced by a simple counterclockwise loop $ \varrho$ encircling only one critical value $h=0$, where $U_{y\epsilon}=B_{y\epsilon}\cap\overline{L}_h$.  
That is, we have
\begin{equation}\label{eq-var-py}
	\mathcal{V}_{{\varrho}y}(\delta_{h}) =-\sum_{j=0}^{N_2-1}(\delta_{h}\cdot \varsigma_{h,j}) \varsigma_{h,j},
\end{equation} 
where $\delta_{h}\in \mathcal{H}_{1}(L_h,\mathbb{Z})$  and $ \varsigma_{h,j}$ is  a cycle vanishing at  $\mathrm{P}_y$ defined by a similar way to that of $ \vartheta_{h,j}$. Namely,   $ \{\varsigma_{h,j}\}$ are the corresponding   cycles  satisfying the part $(\mathrm{II})$ of Theorem \ref{th-vc-inf-0}.

At the end of this section, we study the intersection form between the vanishing cycles in $\mathcal{H}_1(L_h,\mathbb{Z})$ and  have the following lemma.

\begin{lemma}\label{lem-l0-int}
For  system (\ref{m}), if there exists a cycle $\delta_h \in\mathcal{H}_1(L_h,\mathbb{Z})$ such that the intersection number $ \delta_h\cdot \gamma_h =m\neq 0$, then the following statments hold:
\begin{enumerate}
  \item[$\mathrm{(I)}$] For $a_0= 0$ but $a_1\neq 0$,  $\sum_{j=1}^{n-1}  (\delta_h\cdot \gamma_{h,j0})  =m$; for $a_0=a_1=\cdots=a_{N_1-1}= 0$ but $a_{N_1}\neq 0$ with $1<N_1<(n+1)/2$, $\delta_h\cdot \vartheta_{h,0}=m$ and  $\sum_{j=1}^{N_1-1}  (\delta_h\cdot \vartheta_{h,j})  =m$.  
  \item[$\mathrm{(II)}$] For $a_{n+1}= 0$ but $a_n\neq 0$, $\sum_{j=1}^{n-1}  (\delta_h\cdot \gamma_{h,0j}) =m$; for $a_{n+1}=a_n=\cdots=a_{n+2-N_2}= 0$ but $a_{n+1-N_2}\neq 0$ with $1<N_2<(n+1)/2$, $\delta_h\cdot \varsigma_{h,0}=m$ and  $\sum_{j=1}^{N_2-1}  (\delta_h\cdot \varsigma_{h,j})  =m$.    \end{enumerate}
\end{lemma}

\begin{proof}
For the first part of conclusion (I), if $a_0=0$ and $a_1\neq 0$, then we can express the Hamiltonian function as $H(x,y)=y(x+H_{n+1}/y)$.  The level curve $L_0$ can be decomposed into two connected components: the line $C_1=\{(x,y):  y=0\}$, and  $C_2=\{(x,y): x+H_{n+1}/y=0\}$. These two components intersect at  the points $\{(x_j,0),j=0,1,..,n-1\}$, where $x_0=0$ and $x_j $ is defined as shown in the proof of Theorem \ref{th-l0-p}. 

Consider the punctured sphere $S_0:=C_1- \{(x_j,0),j=0,1,..,n-1\}$. Near each punctured point $(x_j,0)$, we can choose a sufficiently small circle $c_j$ centered at $(x_j,0)$, which can be regarded as  the projection of the vanishing cycle $\gamma_{h,j0}$ on $C_1$, such that the interior $D_j$ of $c_j$ containing $(x_j,0)$  satisfies $D_j \cap D_k =\emptyset$ for any $j\neq k$. 
The connected part $U_0=S_0 -\cup_{j=0}^{n-1}{D_j}$ can be deformed to a homeomorphic domain  on $L_h$  whose boundary is  the union of the closed loops representing the  vanishing cycles associated to $\{(x_j,0)\}$.

Let $\delta_0 \subset S_0$ be the limit loop of a series of closed loops on $L_h$ representing $\delta_h$  as $h$ tends to $0$ continuously.
By orienting the sequence of loops
 $\{c_j\}$ in the same orientation to $\gamma_h$,  
    the intersection numbers can be calculated by using the intersection form between $\delta_0$ and the orientational circles $\{c_j\}$. 
   More precisely, we have the the following formula: $\delta_h \cdot \gamma_{h,j0}=\delta_0 \cdot c_j$ for all $j=0,1,...,n-1$, where $\gamma_{h,00}=\gamma_h$.

    Note that in this case the point $\mathrm{P}_{x}$ is  normal on $\overline{L}_h$.
   Thus, as $\delta_0$ crosses from $C_1$ to $C_2$, it must pass through the set of  the points  $\{(x_j,0)\}$, thereby through  $\cup_j{D_j}$, 
   see Figure \ref{fig:int-num} for example. 
   More precisely, on one hand,
   the connected components of $\delta_0$  that contribute  non-zero intersection numbers to  $\delta_0 \cdot c_0$ are only those that 
   enter (or leave) the domain 
  $U_0$ through $c_0$ but leave (or enter) $U_0$ through another $c_j$, with $j\neq 0$.  
   On the other hand,
   if a connected part of $\delta_0$ crosses from  one disc $D_j$, with $j\neq 0$ to another disc $D_k$, with $k\neq 0$, avoiding the disc $D_0$,
   then its contribution is $0$ to the total intersection number between $\delta_0$ and $\sum_{j=1}^{n-1}  \gamma_{h,j0}$ 
   (See, e.g., the dashed lines in Figure \ref{fig:int-num}).
 Therefore, 
  the total intersection number between $\delta_0 $ and  $\cup_{j=1}^{n-1}  c_j$  is equal to $\delta_0 \cdot c_0$, and the following equalities hold:
   $$\sum_{j=1}^{n-1}  \delta_{h} \cdot \gamma_{h,j0}   =\sum_{j=1}^{n-1}  \delta_0 \cdot \gamma_{h,j0} =\delta_0 \cdot c_0=\delta_{h}\cdot \gamma_{h}=m,$$
which  leads to the conclusion. 

For the second part of statement (I), recall the proof of Theorem \ref{th-vc-inf-0}, where
it has been demonstrated that the cycle $ \vartheta_{h,0}$ is homologous to $\gamma_h$ on the Riemann surface of $\overline{L}_h$. 
Consequently, we have  $\delta_h \cdot \vartheta_{h,0}=\delta_h \cdot \gamma_h=m$. 
After the blow-up $\Phi$ we consider the 
 sphere $u=0$. Remove $N_1$   points $(0,v_j)$,  $j=0,1,...,N_1-1$ and choose a sufficiently small orientational loop centered at $(0,v_j)$ for every point $(0,v_j)$, which represents the projection of the vanishing cycle $\vartheta^{*}_{h,j}$ on $u=0$, so that they do not intersect with each other nor contain another one in the interior.   
By  employing the similar technique used in the first part,  we can also obtain $$\sum_{j=1}^{N_1-1}  \delta^{*}_h\cdot \vartheta^{*}_{h,j}  =\delta^{*}_h \cdot \vartheta^{*}_{h,0}=(N_1-1)m,$$
 which implies that 
$$\sum_{j=1}^{N_1-1}  \delta_h\cdot \vartheta_{h,j}  =\delta_h \cdot \vartheta_{h,0}=m.$$

   \begin{figure}[htbp]
\centering
	\includegraphics[]{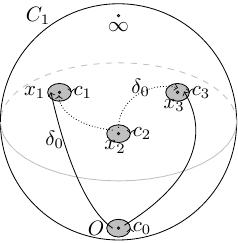}
	\caption{Intersection form on $S_0$ ($n=4$)}\label{fig:int-num}
\end{figure}

The statement  (II) can be proved in the same way. 
\end{proof}

\section{Proof of Theorem \ref{th-main}}

We can now give a detailed proof of  Theorem \ref{th-main}. At the end of the paper, we will provide an alternative  proof for the sufficiency of the conditions using the normal form method.

\begin{proof}[Proof of Theorem \ref{th-main}]

We first prove the necessity.

\noindent {\emph{Necessity}.}

In order to prove the   necessity, we shall show that 	if none of the two conditions in Theorem \ref{th-main} holds, then the origin is not isochronous for   system (\ref{m}). 
By Lemma \ref{lem-noiso-n} and Lemma \ref{lem-nonisol},  it is sufficient  to prove that if there exist two non-zero coefficients $a_{j}$ and $a_{k} $ with two numbers $j,k$ such that $0\leq j< (n+1)/2<k\leq n+1 $, then the origin is not isochronous. Under this hypothesis, each singular point is isolated.

In Theorem \ref{th-non-pole1}, we have shown that the period 1-form $\omega$ does not  possess  any pole at infinity on $\overline{L}_h$ for any $h\neq 0$ (even if $h$ takes another critical value).  This implies that the vanishing cycle $\gamma_h$ cannot be homologous to a zero cycle on the Riemann surface of $\overline{L}_h$. Therefore, there exists a cycle $\gamma'_h\in \mathcal{H}_1(L_h, \mathbb{Z})$ such that the intersection number $\gamma'_h \cdot \gamma_h\neq 0$. By Theorem 2.4 in reference  \cite{Gavri}, we can assume $\gamma'_h$ is also a vanishing cycle associated to a singular point (maybe at infinity). 

 We assert that $\gamma'_h$ must be a vanishing cycle associated to a finite critical point $O'$ of $H(x,y)$ on a level curve $L_{h'}$ with the critical value $h'\neq 0$. The reasons are as follows.
 On one hand,  given that  the origin is non-degenerate, if $\gamma'_h$ vanishes at a point $O'\in\overline{L}_0$ ($O'=O$ or not)   as $h\rightarrow 0$, then the intersection number $\gamma'_h \cdot \gamma_h= 0$, which is in contradiction with the choice of $\gamma'_h$. Thus, $O'\not \in \overline{L}_0$.
  On the other hand, by Theorem \ref{th-lambda} the index $\lambda^{h}(H)=0$ for any atypical value $h\neq 0$. Then by   Theorem 2.4 in \cite{Gavri} again, for any $h\neq 0$, each  cycle in $\mathcal{H}_1(L_h,\mathbb{Z})$ vanishing as $h\rightarrow h'$ can be associated to a finite  critical point of $H(x,y)$ with the critical value $h'$. We have shown that the assertion is indeed true.
 
 From Theorem \ref{th-sim-s}, the polynomial $H(x,y)$ has only simple singularites. Define  
 $$\Delta_{\varrho'}: \mathcal{H}_1(L_h,\mathbb{Z})\rightarrow \mathcal{H}_1(L_h,\mathbb{Z})$$ is the monodromy operator on the first homology group of a generic fiber of $H$  induced by a simple counterclockwise loop $\varrho'$ encircling only one critical value $h'$ with a base point $h\in \mathbb{C}-A_{H}$. Then we have $\Delta_{\varrho'}(\gamma_{h})\cdot \gamma_h \neq 0$. This  can be obtained by repeating the proof of the claim about the crucial non-zero intersection number  occurring in  the proof of Theorem 4.1 in reference \cite{Gavri}. Indeed,   the proof of that claim, rather than the entirety of the theorem,  is also valid for the conditions that are  mildly weaker than those of Theorem 4.1 in \cite{Gavri}. These weaker conditions can be stated as follows: firstly,  the polynomil $H(x,y)$ has only simple singularites; secondly,  the index $\lambda^{h}(H)=0$ for any atypical value $h\neq 0$.
 
 

Let $\delta_h =\Delta_{\varrho'}(\gamma_h)$, $\delta_h\cdot \gamma_h =m \neq 0$, and 
$$\Delta_{\varrho}: \mathcal{H}_1(L_h,\mathbb{Z})\rightarrow \mathcal{H}_1(L_h,\mathbb{Z})$$ be the monodromy operator on the first homology group of a generic fiber of $H$ induced by a simple counterclockwise loop $\varrho$ encircling only one critical value $0$ with the base point $h\in \mathbb{C}-A_{H}$. To show that the origin is not isochronous, it is essential to calculate  $\Delta_{\varrho}(\delta_h)$. 
Although there may be cycles in $\mathcal{H}(L_h,\mathbb{Z})$ represented by the loops encircling only punctured points at infinity, which become trivial when treated as the cycles on the Riemann surface of $\overline{L}_h$, the contribution of those cycles to $\Delta_{\varrho}(\delta_h)$  is zero, since that each of them intersects any other cycle with a zero intersection number. 
Consequently,
the calculation can be divided into the following five cases.

In the first case we assume that $a_0 a_{n+1}\neq 0$. According to Corollary \ref{cor-l0}, the level curve $L_0$ contains a single critical point. In addition, since $\mathrm{P}_{x}, \mathrm{P}_{y} \not \in \overline{L}_h$, by Theorem \ref{th-lambda}, we have $\lambda^{h}(H)=0$ for all $h\in \mathbb{C}$, which means that $H(x,y)$ is a good polynomial as defined in \cite{Gavri}. Then one can use the similar method employed in the proof of Theorem 4.1 in \cite{Gavri} to obtain that 
\begin{equation*}
\begin{array}{lcl}
\Delta_{\varrho}(\delta_h)	&=& \delta_h -(\delta_h\cdot \gamma_h) \gamma_h \end{array}
\end{equation*} 
 by       the Picard-Lefschetz formula at the origin.
 Consequently, 
\begin{equation}\label{eq-delta-1}
\begin{array}{lcl}
\oint_{\Delta_{\varrho}(\delta_h)}\omega	&=& \oint_{\delta_h}\omega -m \oint_{\gamma_h} \omega.
\end{array}
\end{equation} 
 
 The second case is characterised by two dual conditions: $a_0=0$ and $a_1a_{n+1}\neq 0$,  or $a_{n+1}=0$ and $a_0a_{n}\neq 0$.
  We shall focus on the former condition; the latter one can be dealt with in a similar way, with the positions of $x$ and $y$ merely interchanged.
 In this case $\mathrm{P}_y\not\in \overline{L}_h$ and $\lambda^{h}(H)=0$ for any $h$ by Theorem \ref{th-lambda}. It thus follows that  there are no non-trivial cycles vanishing at infinity.
 However, the level curve $L_0$  also contains   $n-1$ finite critical points $ \{(x_j,0),j=1,...,n-1\}$ except the origin as shown in the proof of Theorem \ref{th-l0-p}. Furthermore, each of these points  is non-degenerate due to Theorem \ref{th-sim-s}. Then  the monodromy operator $\Delta_{\varrho}$ can be obtained  from the   variation operations associated to those critical points on $L_0$ that can be calculated by the Picard-Lefschetz formula one by one. Namely, we have 
  \begin{equation*}
\begin{array}{lcl}
\Delta_{\varrho}(\delta_h)	= \delta_h -(\delta_h\cdot \gamma_h) \gamma_h - \sum_{j=1}^{n-1} (\delta_h\cdot \gamma_{h,j0}) \gamma_{h,j0} ,	
\end{array}
\end{equation*} 
and
\begin{equation*}
\begin{array}{lcl}
\oint_{\Delta_{\varrho}(\delta_h)}\omega	= \oint_{\delta_h}\omega -m \oint_{\gamma_h} \omega - \sum_{j=1}^{n-1} (\delta_h\cdot \gamma_{h,j0}) \oint_{\gamma_{h,j0}}\omega .
\end{array}
\end{equation*} 
By Theorem \ref{th-l0-p} and Lemma \ref{lem-l0-int}, one can obtain
\begin{equation}\label{eq-delta-2}
\begin{array}{lcl}
\oint_{\Delta_{\varrho}(\delta_h)}\omega 
&=&\oint_{\delta_h}\omega-m \oint_{\gamma_h} \omega-\left(\sum_{j=1}^{n-1} \delta_h\cdot \gamma_{h,j0}\right) \frac{2\pi \mathrm{i}}{n-1}+o(1)\\ 
&=&\oint_{\delta_h}\omega-m\oint_{\gamma_h} \omega-\frac{2m\pi \mathrm{i}}{n-1}+o(1).
\end{array}
\end{equation}

The third case is that $a_0=a_{n+1}=0$ but $a_1 a_{n} \neq0$. 
   Similarly to the second case, there are no non-trivial cycles vanishing at infinity.
  In this case, however, there are  $2(n-1)$ finite critical points $ \{(x_j,0),j=1,...,n-1\}$ and $ \{(0,y_j),j=1,...,n-1\}$ on $L_0$ except the origin as shown in the proof of Theorem \ref{th-l0-p}. Each of these points is non-degenerate due to Theorem \ref{th-sim-s}. Thus,  the monodromy operator $\Delta_{\varrho}$ can be obtained  from the sum of all variation operations associated to the critical points on $L_0$, each of which can be calculated by the Picard-Lefschetz formula. Namely, we have     
\begin{equation*}
\begin{array}{lcl}
\Delta_{\varrho}(\delta_h)	&=& \delta_h -(\delta_h\cdot \gamma_h) \gamma_h  \\&&- \sum_{j=1}^{n-1} \left((\delta_h\cdot \gamma_{h,j0}) \gamma_{h,j0} + (\delta_h\cdot \gamma_{h,0j}) \gamma_{h,0j}\right)	,
\end{array}
\end{equation*} 
and 
\begin{equation*}
\begin{array}{lcl}
\oint_{\Delta_{\varrho}(\delta_h)}\omega	&=& \oint_{\delta_h}\omega -m \oint_{\gamma_h} \omega  \\&&- \sum_{j=1}^{n-1} \left((\delta_h\cdot \gamma_{h,j0}) \oint_{\gamma_{h,j0}}\omega + (\delta_h\cdot \gamma_{h,0j}) \oint_{\gamma_{h,0j}}\omega\right).
\end{array}
\end{equation*} 
By Theorem \ref{th-l0-p} and Lemma \ref{lem-l0-int} again, one can get that 
\begin{equation}\label{eq-delta-3}
\begin{array}{rcl}
\oint_{\Delta_{\varrho}(\delta_h)}\omega	 &=&\oint_{\delta_h}\omega-m \oint_{\gamma_h} \omega\\&&-\left(\sum_{j=1}^{n-1} \left(\delta_h\cdot \gamma_{h,j0}+\delta_h\cdot \gamma_{h,0j}\right)\right) \frac{2\pi \mathrm{i}}{n-1}+o(1)\\ 
&=&\oint_{\delta_h}\omega-m \oint_{\gamma_h} \omega-\frac{4m\pi \mathrm{i}}{n-1}+o(1).
\end{array}
\end{equation}

The fourth case also contains  two dual conditions: 
\begin{itemize}
  \item $a_0=a_1=\cdots =a_{N_1-1}=0$, $a_{N_1}\neq 0$, $1<N_1<(n+1)/2$ and $a_{n+1}=0$, $a_{n}\neq 0$;
    \item  $a_{n+1}=a_n=\cdots =a_{n+2-N_2}=0$, $a_{n+1-N_2}\neq 0$, $1<N_2<(n+1)/2$ and $a_{0}=0$, $a_{1}\neq 0$.
\end{itemize}
We are only concerned with the former condition (the latter one may be addressed similarly, with only a shift in the positions of $x$ and $y$).
In this case, there are not only $n-1$ finite critical points $\{(0,y_j)\}$ on $L_0$ except the origin, but also a  point $\mathrm{P}_x$ at infinity with  $\lambda^{0}_{\mathrm{P}_x}\neq 0$. For any other point $\mathrm{P}$ at infinity, the index $\lambda^{0}_{\mathrm{P}}= 0$. By Theorem \ref{th-vc-inf-0}, the subgroup $G_{\mathrm{P}_x}$ is non-trivial, i.e., there are  non-trivial cycles vanishing at the point $\mathrm{P}_x$.
Thus the monodromy operator $\Delta_{\varrho}$ can be obtained  from the variation operations associated to those singular points $\{(0,y_j)\}$, which can be calculated by Picard-Lefschetz formula one by one, and the variation operation $\mathcal{V}_{\varrho x}$ at $\mathrm{P}_x$. Namely, we have     
\begin{equation*}
\begin{array}{lcl}
\Delta_{\varrho}(\delta_h)	= \delta_h -(\delta_h\cdot \gamma_h) \gamma_h - \sum_{j=1}^{n-1} ( \delta_h\cdot \gamma_{h,0j})	 \gamma_{h,0j}+\mathcal{V}_{\varrho x}(\delta_h),
\end{array}
\end{equation*} 
and
\begin{equation*}\label{P-L-p}
\begin{array}{lcl}
\oint_{\Delta_{\varrho}(\delta_h)}\omega	= \oint_{\delta_h}\omega -m \oint_{\gamma_h} \omega - \sum_{j=1}^{n-1} (\delta_h\cdot \gamma_{h,0j}) \oint_{\gamma_{h,0j}}\omega+\oint_{\mathcal{V}_{\varrho x}(\delta_h)}\omega .
\end{array}
\end{equation*} 
In addition, one can also obtain that  $$\sum_{j=1}^{n-1} \left( \left(\delta_h\cdot \gamma_{h,0j}\right) \oint_{\gamma_{h,0j}}\omega\right) =\frac{2m \pi \mathrm{i}}{n-1}+o(1),$$
by Theorem \ref{th-l0-p} and Lemma \ref{lem-l0-int}, and that $$ \oint_{\mathcal{V}_{\varrho x}(\delta_h)}\omega=-\frac{2mN_1\pi \mathrm{i}}{N_1-1} +o(1),$$
 by Theorem \ref{th-vc-inf-0}, Lemma \ref{lem-l0-int} and equation \eqref{eq-var-px}.
Therefore, the following equality holds:
  \begin{equation}\label{eq-delta-4}
\begin{array}{c}
\oint_{\Delta_{\varrho}(\delta_h)}\omega	= \oint_{\delta_h}\omega -m \oint_{\gamma_h} \omega - \frac{2m \pi \mathrm{i}}{n-1}-\frac{2mN_1\pi \mathrm{i}}{N_1-1} +o(1).
\end{array}
\end{equation}

In the fifth case, both of the following  two conditions hold simultaneously: $$a_0=a_1=\cdots =a_{N_1-1}=0, \ a_{N_1}\neq 0,\ 1<N_1<\frac{n+1}{2},$$ and $$a_{n+1}=a_n=\cdots =a_{n+2-N_2}=0, \ a_{n+1-N_2}\neq 0, \ 1<N_2<\frac{n+1}{2}.$$ 
 Corollary \ref{cor-l0} indicates that the level curve $L_0$ contains a single critical point which is the origin.  
By Theorem \ref{th-lambda} and  Theorem \ref{th-vc-inf-0},  the  non-trivial cycles vanishing at infinity consist of  those that vanish at $\mathrm{P}_x$ and $\mathrm{P}_y$, respectively, as $h\rightarrow 0$.
In this case, the monodromy operator $\Delta_{\varrho}$ can be obtained  from the variation operator at the origin and two variation operators $\mathcal{V}_{\varrho x}$ and $\mathcal{V}_{\varrho y} $ at  points $\mathrm{P}_x$ and $\mathrm{P}_y$. Consequently, we have
\begin{equation*}
\begin{array}{lcl}
\Delta_{\varrho}(\delta_h)	&=& \delta_h -(\delta_h\cdot \gamma_h) \gamma_h +\mathcal{V}_{\varrho x}(\delta_h)+\mathcal{V}_{\varrho y}(\delta_h),
\end{array}
\end{equation*} 
and 
\begin{equation}\label{eq-delta-5}
\begin{array}{lcl}
\oint_{\Delta_{\varrho}(\delta_h)}\omega	&=& \oint_{\delta_h}\omega -m \oint_{\gamma_h} \omega +\oint_{\mathcal{V}_{\varrho x}(\delta_h)}\omega +\oint_{\mathcal{V}_{\varrho y}(\delta_h)}\omega\\ 
&=& \oint_{\delta_h}\omega -m \oint_{\gamma_h} \omega-\frac{2mN_1\pi \mathrm{i}}{N_1-1}-\frac{2mN_2\pi \mathrm{i}}{N_2-1}+o(1),
\end{array}
\end{equation} 
by Theorem \ref{th-vc-inf-0}, Lemma \ref{lem-l0-int} and equations \eqref{eq-var-px} and \eqref{eq-var-py}.

Suppose the origin is an isochronous center, then the analytic continuation of the period function $T(h)=\oint_{\gamma_h}\omega$ along the loops $\varrho$ and $\varrho'$
gives  that
\begin{equation}\label{eq-conti}
\oint_{\Delta_{\varrho}(\delta_h)}\omega=\oint_{\delta_h}\omega	= \oint_{\gamma_h}\omega=2\pi\mathrm{i}.	
\end{equation}
By applying  equations \eqref{eq-conti} to equations \eqref{eq-delta-1} to \eqref{eq-delta-5}, respectively, it becomes evident that none of them holds. This leads to a contradiction in each case, which proves that the origin cannot be isochronous under the hypothesis that there exist two non-zero coefficients $a_{j}$ and $a_{k} $ with two numbers $j,k$ such that $0\leq j< (n+1)/2<k\leq n+1 $.
.  
The necessity part has been completed. 
\\  
\\
  \noindent {\emph{Sufficiency}.} 
	
	 We first prove the sufficiency of the condition (I). 	 Rewrite $$H(x,y)=xy + y^N \tilde{H}(x,y),$$ where
 $\tilde{H}(x,y)$ is a homogeneous polynomial of degree $n+1-N$, $N> (n+1)/2$, with $\tilde{H}(x,0)=a_{N}x^{n+1-N}\neq 0$.
The  level curve $L_0$ consists of two components given by the following equations respectively:  $$y=0 \quad {\rm{and}}\quad  x+y^{N-1}\tilde{H}(x,y)=0.$$ 
These two components  intersect at a single finite point $(x,y)=(0,0)$ and  at a single  point $\mathrm{P}_x$ at infinity with the projective coordinate $[1:0:0 ]$ in ${\mathbb{CP}^2}$. 

 In the affine chart   $(X,Y)=(1/x,y/x)$,  the point $\mathrm{P}_x$  is changed to the origin that is a singular point of the algebraic curve  defined by
 \begin{equation*}
 	H_{h}^{\ast}(X,Y)=Y^{N}\sum_{j=N}^{n+1} a_j Y^{j-N} +X^{n-1}Y -h X^{n+1}=0.	
 \end{equation*}
In the case where $h\neq 0$, the Newton polygon of the origin consists of two distinct line segments (see the left picture in Figure \ref{NP} below). One of these segments has a slope $-1/2$ and two endpoints $(k,l)=(n-1,1)$ and $(k,l)=(n+1,0)$. Consequently,  it determines a branch of the origin with a Puiseux parameterization  
 \begin{align} \label{1-2-P}
 \begin{array}{c}
X=s,\  Y=s^2(h+h^{N} s^{2N-n-1}+\sum_{j> 2N-n-1} d_j(h) s^j), 
 \end{array} 	
 \end{align}
where $d_j(h)$ is a polynomial of $h$ with zero constant term for all $j>2N-n-1$. Note that this branch continuously tends  to the branch  $Y=0$ of $\overline{L}_0$ as $h\rightarrow 0$. Another segment passing through two endpoints $(k,l)=(n-1,1)$ and $(k,l)=(0,N)$ allows  the corresponding Puiseux parameterization of another branch of the origin 
\begin{equation*} 
 \begin{array}{c}
X=s^{N-1},\  Y=s^{n-1}(d_0+\sum_{j\geq 1} d_j(h) s^j), 
 \end{array} 	
 \end{equation*}
where $d_j(h)$ is also a polynomial of $h$.
In this case, this branch continuously tends  to a branch  of $Y^{N-1}\sum_{j=N}^{n+1} a_j Y^{j-N} +X^{n-1}=0$ on $\overline{L}_0$ as $h \rightarrow 0$, with a Newton polygon as shown in the right picture in Figure \ref{NP} at the origin.

\begin{figure}[htbp]
\centering
	\includegraphics[]{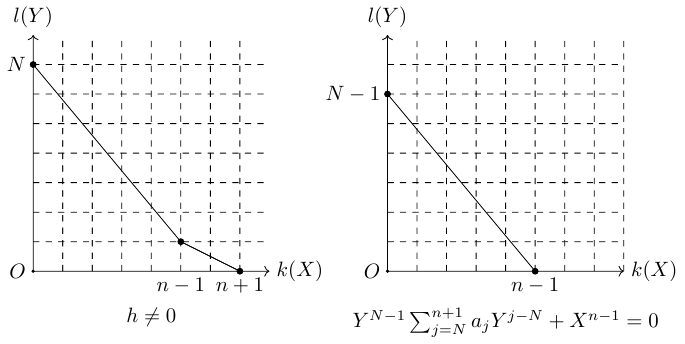}
	\caption{{Newton polygon at the point $[1:0:0]$}}\label{NP}
\end{figure}

In a neighborhood of the point $\mathrm{P}_x$, 
   the fiber bundle defined by $H(x,y)=xy+H_{n+1}(x,y)$ can be extended to $h=0$ and is locally trivial, by Lemma \ref{lem-lambda}.  
A small circle $\delta_0$ centered at $\mathrm{P}_x$  on $y=0$ can be deformed
 to  a small circle $\delta_h$ centered at $\mathrm{P}_x$  on the branch determined by parameterization \eqref{1-2-P} on $L_h$. 
 Noticing that the branch $y=0$ is a topological punctured sphere and choosing another small  circle centered at $O$ on $y=0$ to represent the vanishing cycle $\gamma_0$, we get a topological cylinder $C_0$ with boundaries $\delta_0$ and $\gamma_0$, on which there is no singular points. Then the deformation $C_h$ of  $C_0$ on $L_h$  is  homeomorphic to $C_0$ for any $h$ sufficiently close to $0$. This  means that $\delta_h$ and $\gamma_h$ are boundaries of a topological cylinder $C_h$. In other words,  the cycle $\gamma_h$ is homologous to $\delta_h$ that is a zero cycle on the Riemann surface of $\overline{L}_h$.

Clearly the period $1$-form $\omega $ has no finite poles on $L_h$ for any  $h\in \mathbb{C}-A_{H}$. Consequently, 
 the period function $T(h)$  
	 is equal to the  residue of  $\omega  $ at $\mathrm{P}_x$  on the branch parameterized by \eqref{1-2-P}. 	
	   By substituting this parameterization into the equation (\ref{vf-inf}), we obtain that $$\omega=-\frac{ds}{s+h^{N-1}Na_{N}s^{2N-n}+o(s^{2N-n})}.$$
From the condition $N>(n+1)/2$, it follows that the residue of the $1$-form $\omega$ is equal to $-1$ at $\mathrm{P}_x$ and that $T(h)=2\pi {\mathrm i}$, according to   the orientation of $\gamma_h$. That is,  the singular point  $(x,y)=(0,0)$ is isochronous.

 The sufficiency  of condition (II) of the theorem is analogous to that of condition (I).  Finally, the sufficiency part of the theorem is proved.
	\end{proof}


At the end, we provide an alternative proof of the sufficiency part of Theorem \ref{th-main}.
According to the results in \cite{Gonsa}, the nonlinearities of a planar polynomial systems 
\begin{equation*}
\begin{array}{ccrl}
  	\frac{d x}{d t} & =  & p x+\sum_{k+j>1} a^{(1)}_{kj} x^k y^j,& 
  	\\ 
\frac{d y}{d t}  &= & qy +\sum_{k+j>1} a^{(2)}_{kj} x^k y^j,& \quad p,q\in\mathbb{Z}-\{0\},
\end{array}
\end{equation*}
are  {\emph{admissible}}, if all linear combinations with non-negative integer coefficients (not all zero) of $pk+qj-p$ for $a^{(1)}_{kj}\neq 0$ and  $pk+qj-q$ for $a^{(2)}_{kj}\neq 0$  are nonzero.
When the nonlinearities are admissible, the above system is linearizable near the origin. 

\begin{proof}[Another proof of the  sufficiency of Theorem \ref{th-main}.]
A direct calculation indicates that  the system (\ref{m}) possesses  admissible nonlinearities 
 $\frac{\partial H_{n+1}}{\partial y}$ and $-\frac{\partial H_{n+1}}{\partial x}$. Consequently it is linearizable. For instance,
under the first condition in the theorem, 
we have  $$H_{n+1}(x,y)=\sum_{j>\frac{n+1}{2}}a_j x^{n+1-j}y^j, $$
then the nonlinearities of system (\ref{m}) can be expressed in the following forms: $$\frac{\partial H_{n+1}}{\partial y}=\sum_{j>\frac{n+1}{2}}j a_j x^{n+1-j}y^{j-1}, \  -\frac{\partial H_{n+1}}{\partial x}=-\sum_{j>\frac{n+1}{2}}(n+1-j) a_j x^{n-j}y^{j},$$
such that $p(n+1-j)+q(j-1) -p<0$ and $p(n-j)+q j -q<0$ for all $j>(n+1)/2$, where $p=1$ and  $q=-1$ are the eigenvalues of system (\ref{m}) at the origin. This implies that all linear combinations with non-negative integer coefficients (not all zero) of  $p(n+1-j)+q(j-1) -p$ and $p(n-j)+q j -q$ are negative.
The second condition also implies that the nonlinearities are admissible, since $(n+1-j)-(j-1) -1>0$ and $(n-j)-j +1>0$ for all $j<(n+1)/2$.
\end{proof}

\begin{remark}
It appears challenging to address all the conditions presented in  Theorem \ref{th-main} through the classical normal form methods. However, the topology approach developed herein is able to obtain these conditions effectively, without relying on computer algebra techniques. Moreover, our approach can also provide valuable insights into the polynomial Hamiltonian function, as illustrated in the paper.
\end{remark}



	\section*{Acknowledgement}

The first author was partially supported by the China Scholarship Council (No. 202306780018).
The second author was partially supported by the National Natural Science Foundation of China   (No. 12071006).



\begin{thebibliography}{8}

\bibitem{Amel} V. V. Amel'kin,
Positive solution of one conjecture in the theory
of polynomial isochronous centers of Li$\acute{\rm e}$nard systems,  Differential Equations 57 (2) (2021), 133-138.


\bibitem{AGR} B. Arcet, J. Gine and V.G. Romanovski, Linearizability of planar polynomial Hamiltonian systems,
Nonlinear Analysis: Real World Applications 63 (2022), 103422, 19.

\bibitem{AGV1}V.I. Arnold, S.M. Gusein-Zade, and A.N. Varchenko, Singularities of Differentiable Maps, Volume 1 and 2, Reprint of the 1985 and 1988 Edition, Birkhauser, 2012.

\bibitem{Gonsa} J. Basto-Goncalves, Linearization of resonant vector fields, Trans. Amer. Math. Soc. 362 (12) (2010), 6457-6476. 


\bibitem{zhang}  X. Chen, V.G. Romanovski, and W. Zhang,
Non-isochronicity of the center at the origin in polynomial
Hamiltonian systems with even degree nonlinearities,  Nonlinear Analysis, 68 (9) (2008), 2769-2778.

\bibitem{Chr-Dev} C.J. Christopher and C.J. Devlin, Isochronous centers in planar polynomial systems, SIAM J. Math. Anal. 28 (1997), 162-177.

\bibitem{ChRo} C. Christopher and C. Rousseau, Nondegenerate linearizable centres of complex planar quadratic and symmetric cubic systems in $\mathbb{C}^2$, Publ. Mat. 45 (1) (2001), 95-123.

\bibitem{cima} A. Cima, F. Ma$\tilde{\rm n}$osas, and J. Villadelprat, Isochronicity for several classes of Hamiltonian systems, J. Differential Equations 157 (2) (1999), 373-413.

\bibitem{Fischer} G. Fischer,
Plane algebraic curves,
Translated from the 1994 German original by Leslie Kay.
Student Mathematical Library 15. American Mathematical Society, Providence, RI, 2001.





\bibitem{Gasull}  A. Gasull, A. Guillamon, V. Ma$\tilde{\rm n}$osa and F. Ma$\tilde{\rm n}$osas, The period function for Hamiltonian systems with homogeneous nonlinearities, J. Differential Equations 139 (1997), 237-260.

\bibitem{Gavri}  L. Gavrilov, Isochronicity of plane polynomial Hamiltonian systems, Nonlinearity 10 (1997), 433-448.



\bibitem{llibre} J. Llibre and  V.G. Romanovski,  Isochronicity and linearizability of planar polynomial Hamiltonian systems, J. Differential Equations 259 (5) (2015), 1649-1662.

\bibitem{Man}   F. Ma$\tilde{\rm n}$osas  and J. Villadelprat,  Area-Preserving normalizations for centers of planar
Hamiltonian systems, J. Differential Equations 179 (2) (2002), 625-646.




\end{thebibliography}
\end{document}